\documentclass{amsart}
\usepackage{amsfonts}
\usepackage{amsmath}
\usepackage{amssymb}
\usepackage{amscd}

\input{tcilatex}

\begin{document}
\title[On reaction-diffusion systems.]{Global existence for reaction
diffusion systems with strict balance Law and nonlinearities with non
constant-sign and unlimited polynomial growth .}
\author{Said Kouachi}
\address{Department of Mathematics and Informatics, College of Science and
Technology, University of Abbes Laghrour Khenchela. Algeria.}
\email{kouachi@univ-khenchela.dz}
\thanks{This work was supported by the College of Science and Technology,
University of Abbes Laghrour Khenchela, project code: C00L03UN400120220001.}
\date{December 7, 2022}
\maketitle

\begin{abstract}
The purpose of this paper is to prove global existence of solutions for
general systems of reaction diffusion equations with nonlinearities for
which only two main proprieties hold: Quasi-Positivity and balance law but
with two difficulties: they change sign and with unlimited polynomial
growth. We overcome the first difficulty by fixing the reaction sign after
some time and the second one by using a judicious polynomial Lyapunov
functional.
\end{abstract}

\section{\textbf{Introduction}}

We consider the following reaction-diffusion system 
\begin{equation}
\left\{ 
\begin{array}{l}
\dfrac{\partial u}{\partial t}-a\Delta u=f(t,x,u,v)=-\varphi (t,x,u,v)\text{%
\qquad \qquad in }\mathbb{R}^{+}\times \Omega , \\ 
\\ 
\dfrac{\partial v}{\partial t}-b\Delta v=g(t,x,u,v)=\varphi (t,x,u,v)\text{%
\qquad \qquad in }\mathbb{R}^{+}\times \Omega ,%
\end{array}%
\right.  \label{1.1}
\end{equation}%
with the boundary conditions

\begin{equation}
\frac{\partial u}{\partial \eta }=\frac{\partial v}{\partial \eta }=0\text{%
\qquad \qquad on }\mathbb{R}^{+}\times \partial \Omega ,  \label{1.2}
\end{equation}%
and the initial data

\begin{equation}
u(0,x)=u_{0}(x),\qquad v(0,x)=v_{0}(x)\qquad \text{in}\;\Omega ,  \label{1.3}
\end{equation}%
where $\dfrac{\partial }{\partial \eta }$ denotes the outward normal
derivative on the boundary $\partial \Omega $ of an open bounded domain $%
\Omega \subset \mathbb{R}^{n}$ of class $\mathbb{C}^{1}$ and $a$ and $b$ are
positive constants. The initial data are assumed to be bounded on $\Omega $
and nonnegative. The reaction $\varphi \in C^{1}(\mathbb{R}^{+}\mathbb{%
\times \mathbf{\Omega }\times R}\times \mathbb{R},\mathbb{R})$ isn't
necessarily constant sign with%
\begin{equation}
\ \varphi (t,x,u,0)=\varphi (t,x,0,v)=0,\ \text{for all }u\geq 0,\ v\geq 0,\ 
\text{ }t>0,\ x\in \Omega ,  \label{P}
\end{equation}%
and at most of polynomial growth with respect to the third and fourth
variables:%
\begin{equation}
\text{polynomial-Growth:\ \ }\left\vert \varphi (t,x,s_{1},s_{2})\right\vert
\leq C_{1}+C_{2}\left( s_{1}+s_{2}\right) ^{r},\ \ \ \text{for all }%
s_{1},s_{2}\geq 0,  \label{G}
\end{equation}%
for some $r\geq 1$ and positive constants $C_{1}\ $and $C_{2}$. We are
interested with global existence in time of positive solutions to problem (%
\ref{1.1})-(\ref{1.3}). When the reaction $\varphi $\ is constant sign, then
global existence is immediate and many results have been obtained (see \cite%
{Kou-You} and \cite{Har-You} \cite{Pie-Sch 1993}, \cite{Pie}, ). It is clear
that here, the strict control of mass (called also balance law) is satisfied:%
\begin{equation}
\ \ f(t,x,u,v)+g(t,x,u,v)=0,\ \text{for all }u\geq 0,\ v\geq 0,\ \text{ }%
t>0,\ x\in \Omega ,\ \ \ \ \ \   \label{M}
\end{equation}%
which gives, with homogenous boundary conditions (\ref{1.2}), the total mass
of the solution is invariant.

Since the reaction $\varphi \in C^{1}(\mathbb{R}^{+}\mathbb{\times \mathbf{%
\Omega }\times R}\times \mathbb{R},\mathbb{R})$, then it is locally
Lipschitz and there exists a unique regular solution locally in time in some
maximal interval $\left( 0,T_{\max }\right) $ which satisfies $u,\ v\in
C\left( \left( 0,T\right) ,L^{p}\left( \Omega \right) \right) \cap L^{\infty
}\left( \left( 0,T-\tau \right) \times \Omega \right) $ for all $p>n$ and
all $\tau \in $ $\left( 0,T\right) $ (see for example \cite{Ama 85} and \cite%
{Hen}). Moreover%
\begin{equation}
\text{If }\underset{t\rightarrow T_{\max }}{\lim }\left( \left\Vert
u(t,.)\right\Vert _{L^{\infty }\left( \Omega \right) }+\left\Vert
v(t,.)\right\Vert _{L^{\infty }\left( \Omega \right) }\right) <+\infty ,\ \ 
\text{then }T_{\max }=+\infty .  \label{G.E.}
\end{equation}%
When the reaction is constant sign, the existence in time of a global
solution is trivial : It is the case for example when the reaction $\varphi $%
\ is positive, then by the maximum principle we obtain the uniform
boundedness of the component $u$ and using the strict balance law (\ref{M}),
we deduce the boundedness of $v$ and the global existence in time follows
automatically. Also when the coefficients of diffusion $a$ and $b$ are
equal, global existence is deduced by summing the two equations, using the
positivity of the solutions and applying the maximum principle.

Because of (\ref{G.E.}), this is not so obvious in our case since, in
addition of $a\neq b$, the reaction isn't constant sign and the maximum
principle isn't applicable.

Let us denote by $\overline{t}$ the greats positive time in $\left(
0,T\right) $ such that%
\begin{equation}
\varphi (\overline{t},\overline{x},\overline{u},\overline{v})=0,  \label{1.9}
\end{equation}%
for some $\left( \overline{x},\overline{u},\overline{v}\right) \in \mathbf{%
\Omega }\mathbb{\times R}^{+}\times \mathbb{R}^{+}$. If such $\overline{t}$
doesn't exist, then the reaction $\varphi $ doesn't change sign on $\left(
0,T\right) $ and by the above remarks the global existence becomes.
Otherwise, one of the two following situations is presented:%
\begin{equation}
\text{ }f(t,x,u,v)<0,\text{ for a all }t\in \left( \overline{t},T\right) ,\
x\in \Omega ,\ \ \left( u,v\right) \in \mathbb{R}^{+}\times \mathbb{R}^{+},
\label{1.7.1}
\end{equation}%
or%
\begin{equation}
f(t,x,u,v)>0,\text{ for a all }t\in \left( \overline{t},T\right) ,\ x\in
\Omega ,\ \ \left( u,v\right) \in \mathbb{R}^{+}\times \mathbb{R}^{+}.
\label{1.7.2}
\end{equation}%
This uniform control on the mass (or in mathematical terms on the L1-norm of
the solution) proposes that no blow-up should arise in finite time.
Unfortunately, the situation is not so straightforward. To show how the
situation is difficult, the authors in \cite{Des-Fel-Pie-Vov}, \cite{Fis 15}%
, \cite{CGV19}, \cite{Sou 18}, and others forced the reactions which are at
most quadratic growth to satisfy a Lyapunov structure of LlogL entropy type 
\begin{equation}
\left( \log u\right) f+\left( \log v\right) g\leq 0,\text{ for a all }u>0,\
v>0,  \label{1.8}
\end{equation}%
and obtained only global existence for weak solutions except in some
particular cases.

The reactions terms change sign signifies that none of the equations is nice
in the sense that neither $u$ nor $v$\ is a priori bounded or at least
bounded in some $Lp$-space for $p$ large to apply the well-known
regularizing effect and deduce the global existence of strong solutions in
time for problem (\ref{1.1})-(\ref{1.3}). We should remark that the global
existence of solutions with quadratic nonlinearities was proved by \cite{Sou
18} in the whole space as well as for bounded domains and without using the
control mass condition.

\section{\protect\bigskip \textbf{Notations and preliminary observations}}

It is well-known that to prove global existence of solutions to (\ref{1.1})
(see for example \cite{Hen} and \cite{Har-You}), it suffices to derive a
uniform estimate of $\left\Vert \varphi \right\Vert _{p}$ on $\left[
0,T_{\max }\right[ $ for some $p>n/2$. Here we apply a modified polynomial
Lyapunov functional method analogous to the one used in \cite{Kou Arx 22}
(see also \cite{Kou(ED2001)} \cite{Kou(ED2002)} and \cite{Kou(DPDE)}), to
derive L $^{p}-$bounds of the solution $(u,v)$ of system (\ref{1.1}) and
deduct the global existence in time.

The Lebesgue spaces $\mathbb{L}^{p}(\Omega )$ and $\mathbb{L}^{\infty
}(\Omega )$ and the space $\mathbb{C}\left( \overline{\Omega }\right) $ of
continuous functions on $\Omega $ are respectively endowed with the norms

\begin{equation*}
\left\Vert u\right\Vert _{p}^{p}\text{=}\frac{1}{\left\vert \Omega
\right\vert }\int\limits_{\Omega }\left\vert u(x)\right\vert ^{p}dx,
\end{equation*}%
\begin{equation*}
\left\Vert u\right\Vert _{\infty }\text{=}\underset{x\in \Omega }{max}%
\left\vert u(x)\right\vert .
\end{equation*}%
Since the nonlinear right-hand side of (\ref{1.1}) is continuously
differentiable on $\mathbb{R}^{+}\times \Omega \times \mathbb{R}^{+}\times $ 
$\mathbb{R}^{+}$, then for any initial data in $\mathbb{C}\left( \overline{%
\Omega }\right) $ or $\mathbb{L}^{p}(\Omega ),\;p\in \left( 1,+\infty
\right) $, it is easy to check directly its Lipschitz continuity on bounded
subsets of the domain of a fractional power of the operator 
\begin{equation}
\;\;\left( 
\begin{array}{ll}
-a\Delta & 0 \\ 
0 & -b\Delta%
\end{array}%
\right) .\;\newline
\label{2.2}
\end{equation}%
Under these assumptions, the following local existence result is well-known
(see A. Friedman \cite{Fri}, \cite{Hen}, \cite{Smo} and \cite{Rot}).

\begin{proposition}
The system (1.1)-(1.3) admits a unique, classical solution $(u,v)$ on $%
[0,T_{\max }[\times \Omega $. If $T_{\max }<\infty $ then 
\begin{equation}
\underset{t\nearrow T_{\max }}{\lim }\left\{ \left\Vert u(t,.)\right\Vert
_{\infty }+\left\Vert v(t,.)\right\Vert _{\infty }\right\} =\infty , 
\tag*{(2.3)}
\end{equation}%
where $T_{\max }$ denotes the eventual blowing-up time in $\mathbb{L}%
^{\infty }(\Omega ).$
\end{proposition}

Before the statement of the results, let us define for a fixed integer $%
p\geq 1,$ the following polynomial functional which is of a great interest
in the following%
\begin{equation}
t\mapsto L(t)=\int\limits_{\Omega }H_{p}\left( u(t,x),v(t,x)\right) dx,
\label{2.4}
\end{equation}%
where%
\begin{equation}
H_{p}\left( u,v\right) =\overset{p}{\underset{i=0}{\sum }}C_{p}^{i}\theta
_{i}u^{i}v^{p-i},  \label{2.5}
\end{equation}%
where the coefficient $C_{p}^{i}$ is given by the formula%
\begin{equation*}
C_{p}^{i}=\frac{p!}{i!.(p-i)!}
\end{equation*}%
which is defined in terms of the factorial function $n!$.

The finite sequence $\left\{ \theta _{i}\right\} $ of positive terms is
defined as follows%
\begin{equation}
\theta _{i}=c^{i+1}K^{i^{2}},\ i=0,\ 1,\ 2,...,  \label{2.6.1}
\end{equation}%
when the nonlinearity satisfies (\ref{1.7.1}) and%
\begin{equation}
\theta _{i}=C^{i+1}K^{i^{2}},\ i=0,\ 1,\ 2,...,  \label{2.6.2}
\end{equation}%
when the nonlinearity satisfies (\ref{1.7.2}), where the constants $c$ and $C
$ satisfy the following inequality 
\begin{equation*}
cK^{2p+1}<1<CK,
\end{equation*}%
and $K$ is any positive constant satisfying%
\begin{equation}
K^{2}>\frac{\left( a+b\right) ^{2}}{4ab}.  \label{2.8}
\end{equation}

\section{\protect\bigskip \textbf{Statement and proof of the main result}}

The main result of the paper is the following

\begin{theorem}
Suppose that the nonlinearity $\varphi \in C^{1}(\mathbb{R}^{+}\mathbb{%
\times \mathbf{\Omega }\times R}\times \mathbb{R},\mathbb{R})$ is at most
polynomial growth with respect to the third and fourth variables and
satisfes condition (\ref{P}) then the functional given by (\ref{2.4})-(\ref%
{2.5}) is uniformly bounded on the interval $\left( 0,T_{\max }\right) $ for
all positive solutions of problem (\ref{1.1})-(\ref{1.3})
\end{theorem}

\begin{proof}
By differentiating $L$ with respect to $t$, we get%
\begin{equation}
L^{\prime }(t)=\int\limits_{\Omega }\overset{p}{\underset{i=1}{\sum \ }}%
\left( \;iC_{p}^{i}\theta _{i}u^{i-1}v^{p-i}\right) \frac{\partial u}{%
\partial t}dx+\int\limits_{\Omega }\overset{p-1}{\underset{i=0}{\sum \ }}%
\left( (p-i)C_{p}^{i}\theta _{i}u^{i}v^{p-i-1}\right) \frac{\partial v}{%
\partial t}dx.  \label{3.5}
\end{equation}%
Using the following well-known relation%
\begin{equation}
iC_{p}^{i}=pC_{p-1}^{i-1},\text{\ for\ all\ }\;i=1,...,p,  \label{3.6}
\end{equation}%
and interchanging the indexes, we get%
\begin{eqnarray*}
L^{\prime }(t) &=&p\overset{p-1}{\underset{i=0}{\sum }}C_{p-1}^{i}\int%
\limits_{\Omega }u^{i}v^{p-1-i}\left[ \left( a\theta _{i+1}\Delta u+b\theta
_{i}\Delta v\right) +\left( \theta _{i+1}f(u,v)+\theta _{i}g(u,v)\right) %
\right] dx \\
&=&I+J,
\end{eqnarray*}%
where%
\begin{eqnarray}
I &=&p\overset{p-1}{\underset{i=0}{\sum }}\int\limits_{\Omega
}C_{p-1}^{i}u^{i}v^{p-1-i}\left( a\theta _{i+1}\Delta u+b\theta _{i}\Delta
v\right) dx,  \label{3.7} \\
&&  \notag
\end{eqnarray}%
and%
\begin{eqnarray}
J &=&p\overset{p-1}{\underset{i=0}{\sum }}C_{p-1}^{i}\int\limits_{\Omega
}\left( -\theta _{i+1}+\theta _{i}\right) f(t,x,u\left( t,x\right) ,v\left(
t,x\right) )u^{i}v^{p-1-i}dx.  \label{3.8} \\
&&  \notag
\end{eqnarray}%
The Green's formula applied to the integral $I$ gives%
\begin{eqnarray*}
I &=&-p\overset{p-1}{\underset{i=0}{\sum }}C_{p-1}^{i}\int\limits_{\Omega }%
\left[ a\nabla \left( \theta _{i+1}u^{i}v^{p-1-i}\right) \nabla u+b\nabla
\left( \theta _{i}u^{i}v^{p-1-i}\right) \nabla v\right] dx \\
&&
\end{eqnarray*}

Using (\ref{3.6}) another time and interchanging the indexes, the integral $%
I $ becomes 
\begin{equation}
I=-p(p-1)\overset{p-2}{\underset{i=0}{\sum }}C_{p-2}^{i}\
\int\limits_{\Omega }u^{i}v^{p-2-i}\left( a\theta _{i+2}\left\vert \nabla
u\right\vert ^{2}+\left( a+b\right) \theta _{i+1}\nabla u\nabla v+b\theta
_{i}\left\vert \nabla v\right\vert ^{2}\right) dx.  \label{3.9}
\end{equation}

We shall show that%
\begin{equation}
I\leq 0;\ \ J\leq 0\ \ \text{in the interval }\left( \overline{t},T_{\max
}\right) .  \label{3.10}
\end{equation}

Clearly, from (\ref{2.6.1}) we have%
\begin{equation}
\frac{\theta _{i}\theta _{i+2}}{\theta _{i+1}^{2}}=K^{2},\ i=0,1,...p-2,
\label{2.9}
\end{equation}%
which implies that the quadratic form with respect $\nabla u$ and $\nabla v$
under the integral $I$, is non positive. So we have

\begin{equation}
I\leq 0,\text{ for all }t\in \left( \overline{t},T_{\max }\right) .
\label{3.11}
\end{equation}%
For the integral $J$, two cases arise: Either (\ref{1.7.1}) otherwise (\ref%
{1.7.2}).

In the first case we choose%
\begin{equation*}
\theta _{i}=c^{i+1}K^{i^{2}},\ i=0,\ 1,\ 2,...,
\end{equation*}%
where the constant $c$ is chosen in such away that the sequence $\left\{
\theta _{i}\right\} $ is decreasing. That is 
\begin{equation}
cK^{2p+1}<1,  \label{2.7.1}
\end{equation}%
where%
\begin{equation*}
K^{2}>\frac{\left( a+b\right) ^{2}}{4ab}.
\end{equation*}

Then%
\begin{equation*}
-\theta _{i+1}+\theta _{i}\geq 0,\ i=0,\ 1,\ 2,....
\end{equation*}%
and since from (\ref{1.7.1}) the reaction $f$ \ is non positive, we conclude
that%
\begin{equation}
J\leq 0,\text{ for all }t\in \left( \overline{t},T_{\max }\right) .
\label{2.3}
\end{equation}%
Combining (\ref{2.3}) and (\ref{3.11}), we get (\ref{3.10}).

In the second case, we choose a constant $C>\frac{1}{K}$ and 
\begin{equation*}
\theta _{i}=C^{i+1}K^{i^{2}},\ i=0,\ 1,\ 2,...
\end{equation*}%
Clearly, the sequence $\left\{ \theta _{i}\right\} $ is increasing.
Following the same reasoning as in the first case, we obtain another time (%
\ref{2.3}) and (\ref{3.11}).

Let's recap: In the two cases (\ref{1.7.1}) and (\ref{1.7.2}), we have
proved that the functional $L(t)$ is decreasing on the interval $\left( 
\overline{t},T_{\max }\right) $. That is%
\begin{equation*}
L(t)\leq L(\overline{t}),\text{ for all }t\in \left( \overline{t},T_{\max
}\right) .
\end{equation*}%
obviously we conclude that $L(t)$ is uniformly bounded in the interval $%
\left( \overline{t},T_{\max }\right) $. In the remaining interval $\left( 0,%
\overline{t}\right) ,$ $u$ and $v$ are in $L^{\infty }\left( \left[ 0,T-\tau %
\right] \times \Omega \right) $ for all $\tau \in \left( 0,T\right) $ since
they are classical and so they are uniformly bounded. Finally, we can say
that our functional is uniformly bounded in the hole interval of local
existence $\left( 0,T_{\max }\right) $: That is there exists a positive
constant $M$ independent of the time such that%
\begin{equation*}
L(t)\leq M,\text{ for all }t\in \left( 0,T_{\max }\right) .
\end{equation*}%
This ends the proof of the Theorem.
\end{proof}

Since the coefficients $\theta _{i}$ in the functional $L(t)$ are bounded
below, we can find another constant $R$ such that%
\begin{equation*}
\int\limits_{\Omega }\left( u+v\right) ^{p}dx\leq R.L(t),\text{ for all }%
t\in \left( 0,T_{\max }\right) .
\end{equation*}%
The two above inequalities show that the solutions of problem (\ref{1.1})-(%
\ref{1.3}) are $L^{p}\left( \Omega \right) $ for all $p\geq 1$ and since the
nonlinearity $\varphi \in C^{1}(\mathbb{R}^{+}\mathbb{\times \mathbf{\Omega }%
\times R}\times \mathbb{R},\mathbb{R})$ and at most polynomial growth with
respect to $u$ and $v$, then by the preliminary remarks\ we deduce the
global existence in time of the solutions, that is $T_{\max }=+\infty $.

We have proved the following

\begin{corollary}
Suppose that the nonlinearity  $\varphi \in C^{1}(\mathbb{R}^{+}\mathbb{%
\times \mathbf{\Omega }\times R}\times \mathbb{R},\mathbb{R})$ is at most of polynomial growth with respect to the third and fourth variables, then all
solutions of system (\ref{1.1}) with nonnegative uniformly bounded initial
data and homogeneous Neumann boundary conditions exist globally in time.
\end{corollary}\label{Global}

\begin{remark}
The above corollary remains true for more general  boundary conditions 
\begin{equation}
\left\{ 
\begin{array}{l}
\lambda _{1}u+(1-\lambda _{1})\partial _{\eta }u=0,\ \ \ \ \  \\ 
\ \ \ \ \ \ \ \ \ \ \ \ \ \ \ \ \ \ \ \ \ \ \ \ \ \ \ \ \ \ \ \ \ \ \ \ \ \
\ \ \ \ \text{\ on }\mathbb{R}^{+}\times \partial \Omega , \\ 
\lambda _{2}v+(1-\lambda _{2})\partial _{\eta }v=0,\text{\ } \\ 
,\ \ \ \ \ \ \ \ \ \ \ \ \ \ \ \ \ \ \ \ \ 
\end{array}%
\right.   \label{3.D.C.R.1.4}
\end{equation}%
where, for $\;i=1,\;2$, different type of boundary conditions are imposed:
homogeneous Robin type ($0<\lambda _{i}<1$) or homogeneous Neumann type ($%
\lambda _{i}=0$) or homogeneous Dirichlet type ($1-\lambda _{i}=0,$). Also a
mixture of homogeneous Dirichlet with homogeneous Robin boundary conditions (%
$1-\lambda _{i}=0,\;i=1$ or$\;2$ and $0<\lambda _{j}<1,,\;j=1,\;2$, with $%
i\neq j$) can be assumed.
\end{remark}

\section{Applications}

\subsection{\protect\bigskip The first case}

We begin by the following system%
\begin{equation}
\left\{ 
\begin{array}{l}
\dfrac{\partial u}{\partial t}-a\Delta u=-c(t,x)\psi (u,v),\text{\qquad
\qquad in }\mathbb{R}^{+}\times \Omega , \\ 
\\ 
\dfrac{\partial v}{\partial t}-b\Delta v=c(t,x)\psi (u,v),\text{\qquad
\qquad in }\mathbb{R}^{+}\times \Omega ,%
\end{array}%
\right.  \label{1.4.1}
\end{equation}

where $c(t,x)\in C(\mathbb{R}^{+}\mathbb{\times }\Omega ,\mathbb{R})$ isn't
necessarily constant sign with 
\begin{equation}
\psi (0,v)=\psi (u,0)=0,\text{ for a all }u\geq 0,\ v\geq 0,  \label{1.5}
\end{equation}%
which assures the positivity of the solution on $\Omega $ at all time. When $%
c(t,x)$ is independent of the time, for example 
\begin{equation*}
c(x)<0\text{ in }\left( -1,0\right) ,\ \ c(0)=0,\ \ c(x)>0\text{ in }\left(
0,1\right) ,
\end{equation*}%
the authors in \cite{Pie-Sch 1993}, showed that the solutions are locally
uniformly bounded in $L_{\infty }\left( \left[ 0,\infty )\times \right(
0,1\right) $ and $L^{\infty }\left( \left[ 0,\infty )\times \right(
-1,0\right) $ but they didn't prove global existence. It can arrive that, in
some very special cases, system (\ref{1.1}) with appropriate nonhomogeneous
Dirichlet conditions presents plow-up in finite time (see for example \cite[
22]{Pie-Sch 22}). Recently, in the same case S. Kouachi \cite{Kou Arx 22}\ \
proved global existence by using the same functional (\ref{2.4}) below. To
our knowledge, the question with general $c(t,x)$ is still\textbf{\ }an open
question in all space dimensions (see \cite{Pie}) until now. According to
Corollary \ref{Global}

\begin{proposition}
Suppose that the nonlinearity $\psi (u,v)\in C^{1}(\mathbb{R}^{+}\times 
\mathbb{R}^{+},\mathbb{R})$ is at most of polynomial growth with respect to $%
u$ and $v$, satisfying (\ref{1.5}) and that $c(t,x)\in C(\mathbb{R}^{+}%
\mathbb{\times }\Omega ,\mathbb{R})$ isn't necessarily constant sign, then
all solutions of   
\begin{equation}
\left\{ 
\begin{array}{l}
\frac{\partial u}{\partial t}-a\Delta u=f(t,x,u,v)=-c(t,x)\psi (u,v)\text{%
\qquad \qquad in }\mathbb{R}^{+}\times \Omega , \\ 
\frac{\partial v}{\partial t}-b\Delta v=g(t,x,u,v)=c(t,x)\psi (u,v)\text{%
\qquad \qquad in }\mathbb{R}^{+}\times \Omega , \\ 
\frac{\partial u}{\partial \eta }=\frac{\partial v}{\partial \eta }=0\text{%
\qquad \qquad on }\mathbb{R}^{+}\times \partial \Omega , \\ 
u(0,x)=u_{0}(x)\geq 0,\qquad v(0,x)=v_{0}(x)\geq 0,\qquad \text{in}\;\Omega ,%
\end{array}%
\right.   \label{3.12}
\end{equation}%
are global.
\end{proposition}

\begin{proof}
Since the reaction $c(t,x)\psi (u,v)$ satisfies all conditions of corollary %
\ref{Global}, then obviously the solutions of problem (\ref{3.12}) are
global.
\end{proof}

\subsection{Coupled reversible chemical reactions}

\bigskip The special case with the following system%
\begin{equation}
\left\{ 
\begin{array}{l}
\dfrac{\partial u}{\partial t}-a\Delta u=-\ h_{1}u^{l}v^{q}+h_{2}u^{r}v^{s},%
\text{\qquad \qquad in }\mathbb{R}^{+}\times \Omega , \\ 
\\ 
\dfrac{\partial v}{\partial t}-b\Delta v=\ h_{1}u^{l}v^{q}-h_{2}u^{r}v^{s},%
\text{\qquad \qquad in }\mathbb{R}^{+}\times \Omega ,%
\end{array}%
\right.  \label{1.6.1}
\end{equation}%
which describes the following reversible chemical reaction%
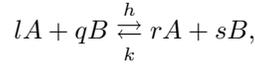
\begin{equation*}
lA+qB\overset{h}{\underset{k}{\rightleftarrows }}rA+sB,
\end{equation*}%
\ where $u$ and $v$ are the concentrations of the reactants $A$ and $B$
respectively. The authors in \cite{Kou(DPDE)}\ and lately in \cite{Pie}
obtained with a slight difference, under restrictive conditions on the
orders of the reactants, the following partial result

\begin{equation}
\left\{ 
\begin{array}{l}
l+q\leq 1\text{ or }r+s\leq 1, \\ 
\text{or }r+s>l+q>1\text{ and }l-r<sl-qr<s-q, \\ 
\text{or }l+q>r+s>1\text{ and }s-q<sl-qr<l-r.%
\end{array}%
\right.  \label{2.11}
\end{equation}%
System (\ref{1.6.1}) is on the form (\ref{1.4.1}) where $c(t,x)$ is constant
sign but $\psi (u,v)$ doesn't 
\begin{equation*}
c(t,x)\equiv 1\text{ and }\psi (u,v)=h_{1}u^{l}v^{q}-h_{2}u^{r}v^{s}.
\end{equation*}%
Corollary \ref{Global} is always applicable in this case and we have the
following Proposition

\begin{proposition}
\bigskip The solutions of problem%
\begin{equation}
\left\{ 
\begin{array}{l}
\frac{\partial u}{\partial t}-a\Delta
u=f(t,x,u,v)=-h_{1}u^{l}v^{q}+h_{2}u^{r}v^{s}\text{\qquad \qquad in }\mathbb{%
R}^{+}\times \Omega , \\ 
\frac{\partial v}{\partial t}-b\Delta
v=g(t,x,u,v)=h_{1}u^{l}v^{q}-h_{2}u^{r}v^{s}\text{\qquad \qquad in }\mathbb{R%
}^{+}\times \Omega , \\ 
\frac{\partial u}{\partial \eta }=\frac{\partial v}{\partial \eta }=0\text{%
\qquad \qquad on }\mathbb{R}^{+}\times \partial \Omega , \\ 
u(0,x)=u_{0}(x)\geq 0,\qquad v(0,x)=v_{0}(x)\geq 0,\qquad \text{in}\;\Omega ,%
\end{array}%
\right.   \label{3.13}
\end{equation}%
are global.
\end{proposition}

\begin{proof}
\bigskip The proof is an immediate consequence of Corollary \ref{Global}.
\end{proof}

\subsection{Extension to $m\times m$ reversible chemical systems}

\subsubsection{\protect\bigskip Tripled reversible chemical reactions}

Let's consider the following tripled reaction-diffusion system%
\begin{equation}
\partial _{t}u-d_{1}\Delta
u=a_{1}u^{p_{1}}v^{q_{1}}w^{r_{1}}-a_{2}u^{p_{2}}v^{q_{2}}w^{r_{2}},\ \ \ \
\ \ \ \ \ \ \ \ \ \ \ \ \ \ \ \ \ \ \   \label{3.D.C.R.1.1}
\end{equation}%
\begin{equation}
\partial _{t}v-d_{2}\Delta
v=a_{1}u^{p_{1}}v^{q_{1}}w^{r_{1}}-a_{2}u^{p_{2}}v^{q_{2}}w^{r_{2}},\ \ \
\;\;\text{in }\mathbb{R}^{+}\times \Omega ,  \label{3.D.C.R.1.2}
\end{equation}%
\begin{equation}
\partial _{t}w-d_{3}\Delta
w=-a_{1}u^{p_{1}}v^{q_{1}}w^{r_{1}}+a_{2}u^{p_{2}}v^{q_{2}}w^{r_{2}},\ \ \ \
\ \ \ \ \ \ \ \ \ \ \ \ \ \ \   \label{3.D.C.R.1.3}
\end{equation}%
with Neumann or Dirichlet homogenous boundary conditions and initial data
nonnegative and uniformly bounded on $\Omega $, where the constants $%
a_{i},p_{i},q_{i},r_{i},\ \ i=1,2$, are positive. System (\ref{3.D.C.R.1.1}%
)-(\ref{3.D.C.R.1.3}) describes a model for the following reversible
chemical reaction%
\begin{equation}
p_{1}U+q_{1}V+r_{1}W\rightleftarrows p_{2}U+q_{2}V+r_{2}W.
\label{3.D.C.R.1.7}
\end{equation}%
In \cite{Kou Hal 22}, we obtained under the following very restrictive
conditions on nonlinearities growths, global existence:

\begin{equation}
p_{i}+q_{i}+r_{i}\leq 1,\ \text{ }i=1,\ 2,  \label{3.D.C.R.3.1}
\end{equation}%
or%
\begin{eqnarray}
p_{j}+q_{j}+r_{j} &>&p_{i}+q_{i}+r_{i}>1,\ \ \ \ i=1,\ j=2\ \text{or}\ i=2,\
j=1  \notag \\
&&and  \label{3.D.C.R.3.2.1} \\
&&\left\{ 
\begin{array}{l}
p_{i}-p_{j}<\left( p_{i}q_{j}-p_{j}q_{i}\right) +\left(
p_{i}r_{j}-p_{j}r_{i}\right) , \\ 
\text{ }q_{i}-q_{j}<\left( q_{i}p_{j}-q_{j}p_{i}\right) +\left(
q_{i}r_{j}-q_{j}r_{i}\right) , \\ 
\text{ }r_{i}-r_{j}<\left( r_{i}p_{j}-r_{j}p_{i}\right) +\left(
r_{i}q_{j}-r_{j}q_{i}\right) .%
\end{array}%
\right.  \notag
\end{eqnarray}

The application of Corollary \ref{Global} gives the following proposition

\begin{proposition}
\bigskip The solutions of system (\ref{3.D.C.R.1.1})-(\ref{3.D.C.R.1.3})
with homogenous Neumann boundary
conditions and positive initial data are global.
\end{proposition}

\begin{proof}
We apply Corollary \ref{Global} separately to the following u-w and v-w
systems :%
\begin{equation}
\left\{ 
\begin{array}{l}
\frac{\partial u}{\partial t}-a\Delta
u=a_{1}u^{p_{1}}v^{q_{1}}w^{r_{1}}-a_{2}u^{p_{2}}v^{q_{2}}w^{r_{2}}\text{%
\qquad \qquad in }\mathbb{R}^{+}\times \Omega , \\ 
\frac{\partial w}{\partial t}-c\Delta
w=-a_{1}u^{p_{1}}v^{q_{1}}w^{r_{1}}+a_{2}u^{p_{2}}v^{q_{2}}w^{r_{2}}\text{%
\qquad \qquad in }\mathbb{R}^{+}\times \Omega , \\ 
\frac{\partial u}{\partial \eta }=\frac{\partial w}{\partial \eta }=0\text{%
\qquad \qquad on }\mathbb{R}^{+}\times \partial \Omega , \\ 
u(0,x)=u_{0}(x)\geq 0,\qquad w(0,x)=w_{0}(x)\geq 0,\qquad \text{in}\;\Omega ,%
\end{array}%
\right.  \label{3.14 u-w}
\end{equation}%
and%
\begin{equation}
\left\{ 
\begin{array}{l}
\frac{\partial v}{\partial t}-b\Delta
v=a_{1}u^{p_{1}}v^{q_{1}}w^{r_{1}}-a_{2}u^{p_{2}}v^{q_{2}}w^{r_{2}}\text{%
\qquad \qquad in }\mathbb{R}^{+}\times \Omega , \\ 
\frac{\partial w}{\partial t}-c\Delta
w=-a_{1}u^{p_{1}}v^{q_{1}}w^{r_{1}}+a_{2}u^{p_{2}}v^{q_{2}}w^{r_{2}}\text{%
\qquad \qquad in }\mathbb{R}^{+}\times \Omega , \\ 
\frac{\partial v}{\partial \eta }=\frac{\partial w}{\partial \eta }=0\text{%
\qquad \qquad on }\mathbb{R}^{+}\times \partial \Omega , \\ 
v(0,x)=v_{0}(x)\geq 0,\qquad w(0,x)=w_{0}(x)\geq 0,\qquad \text{in}\;\Omega ,%
\end{array}%
\right.  \label{3.14 v-w}
\end{equation}%
respectively. Each of the two systems (\ref{3.14 u-w}) and (\ref{3.14 v-w})
presents the same situation as system (\ref{3.13}) . Corollary \ref{Global}
is immediately applicable to get separately bounds of $u$ and $w$ then those
of $v$ and $w$ together. Global existence of system (\ref{3.D.C.R.1.1})-(\ref%
{3.D.C.R.1.3}) becomes automatically.
\end{proof}

\subsubsection{General reversible chemical reactions}

Let's consider the following general reversible chemical reaction%
\begin{equation}
\underset{i\in I}{\tsum }n_{i}R_{i}\overset{h}{\underset{k}{\rightleftarrows 
}}\underset{j\in J}{\tsum }n_{j}R_{j}\text{ \ \ with }I\cup J=\left\{
1,...,p\right\} \text{ \ \ and }I\cap J=\left\{ \emptyset \right\} ,
\label{3.2}
\end{equation}%
where $n_{1},n_{2},...,n_{p}$ are respectively numbers of molecules $%
R_{1},R_{2},...,R_{p}$ taking part in the reaction, the constants $h$ and $l$
depend on the temperature, the position $x$ and the time $t$.\newline
The application of the law of conservation of mass and the second law of
Fick (flux) gives the following Super-quadratic reaction-diffusion system%
\begin{equation}
\left\{ 
\begin{array}{c}
n_{k}\dfrac{\partial c_{k}}{\partial t}=\nabla .d_{k}\nabla c_{k}-h\underset{%
i\in I}{\prod }c_{i}^{n_{i}}+l\underset{j\in J}{\prod }c_{j}^{n_{j}},\ \
k\in I, \\ 
n_{k}\dfrac{\partial c_{k}}{\partial t}=\nabla .d_{k}\nabla c_{k}+h\underset{%
i\in I}{\prod }c_{i}^{n_{i}}-l\underset{j\in J}{\prod }c_{j}^{n_{j}},\ \
k\in J,%
\end{array}%
\right.  \label{3.2.1}
\end{equation}%
where $c_{1},c_{2},...,c_{p}$ represent respectively the concentrations of $%
R_{1},R_{2},...,R_{p}$ and $n_{1},n_{2},...,n_{p}$ are positive constants
called orders of $R_{1},R_{2},...,R_{p}$ respectively. The constants $h$ and 
$l$ are positive. According to Corollary \ref{Global} the following result
concerning system (\ref{3.2.1})

\begin{corollary}
Solutions of system (\ref{3.2.1})\ with positive uniformly bounded initial
data, exist for all $t>0$.
\end{corollary}

\begin{proof}
System (\ref{3.2.1}) satisfies all conditions of system (\ref{3.15 general
i,j}) below, then global existence occurs.
\end{proof}

\section{The general case of reaction diffusion systems}

\bigskip Generally we can apply our results to systems of $m$ reaction
diffusion equations on the form%
\begin{equation}
\left\{ 
\begin{array}{l}
\forall \ j=1,...,m \\ 
\frac{\partial u_{j}}{\partial t}-d_{j}\Delta u_{j}=f_{j}(t,x,u),\text{%
\qquad on }\mathbb{R}^{+}\times \Omega , \\ 
\frac{\partial u_{j}}{\partial \eta }=0,\text{\qquad on }\mathbb{R}%
^{+}\times \partial \Omega , \\ 
u_{j}(0,x)=u_{j0}(x)\geq 0,\ \ \ \text{on }\Omega ,%
\end{array}%
\right.  \label{3.15 general}
\end{equation}%
with nonlinearities at most polynomially growth and satisfying the following
partial strict control of mass 
\begin{equation}
\forall \text{ }i=1,...,m,\ \exists \text{ }j\neq i\text{ such that }%
f_{i}(t,x,u)+f_{j}(t,x,u)=0,\text{ \ \ on }\mathbb{R}^{+}\times \Omega
\times \mathbb{R}^{n}\text{.}  \label{(M)i,j}
\end{equation}%
By applying Corollary \ref{Global} to any pair of equations:%
\begin{equation}
\left\{ 
\begin{array}{l}
\forall \ i,\ j=1,...,m\ \text{with}\ i\neq j \\ 
\begin{array}{l}
\left\{ 
\begin{array}{c}
\frac{\partial u_{i}}{\partial t}-d_{i}\Delta u_{i}=f_{i}(t,x,u),\text{%
\qquad on }\mathbb{R}^{+}\times \Omega , \\ 
\frac{\partial u_{j}}{\partial t}-d_{j}\Delta u_{j}=f_{j}(t,x,u),\text{%
\qquad on }\mathbb{R}^{+}\times \Omega ,%
\end{array}%
\right. \\ 
\frac{\partial u_{i}}{\partial \eta }=\frac{\partial u_{j}}{\partial \eta }%
=0,\text{\qquad on }\mathbb{R}^{+}\times \partial \Omega , \\ 
u_{i}(0,x)=u_{i0}(x)\geq 0,\ \ \ u_{j}(0,x)=u_{j0}(x)\geq 0,\ \ \ \text{on }%
\Omega ,%
\end{array}%
\end{array}%
\right.  \label{3.15 general i,j}
\end{equation}%
with nonlinearities satisfying the partial strict control of mass (\ref%
{(M)i,j}), we can write the following

\begin{proposition}
Positive solutions of the general system of m reaction diffusion equations (%
\ref{3.15 general}) with nonlinearities at most polynomial growth and
satisfying the partial strict control of mass (\ref{(M)i,j}) are global.
\end{proposition}\label{General}

\begin{proof}
By applying Corollary \ref{Global} to a pair of equations (\ref{3.15 general
i,j}) with nonlinearities satisfying the strict partial control of mass (\ref%
{(M)i,j}), we obtain bounds of the two components $u_{i}$ and $u_{j}$.
Finally, bounds are obtained of all components of the general system (\ref%
{3.15 general}).
\end{proof}

\begin{remark}
The super-quadratic reaction-diffusion system (\ref{3.2.1}) describing the
general reversible chemical reaction (\ref{3.2}) is a particular case of \
system (\ref{3.15 general}) with nonlinearities at most polynomial growth
and satisfying the strict control of mass by pairs(\ref{(M)i,j}). Global
existence of solutions becomes easy via Corollary \ref{Global}.
\end{remark}

\end{document}